\documentclass{article}

\usepackage{arxiv}

\usepackage[utf8]{inputenc} 
\usepackage[T1]{fontenc}    
\usepackage{hyperref}       
\usepackage{url}            
\usepackage{booktabs}       
\usepackage{amsfonts}       
\usepackage{nicefrac}       
\usepackage{microtype}      
\usepackage{lipsum}		
\usepackage{graphicx}
\usepackage{natbib}
\usepackage{doi}

\usepackage{amsthm,amsmath}
\newtheorem{theorem}{Theorem}[section]
\newtheorem{corollary}[theorem]{Corollary}
\newtheorem{lemma}[theorem]{Lemma}

\theoremstyle{definition}
\newtheorem{definition}[theorem]{Definition}
\theoremstyle{remark}

\newtheorem*{example}{Example}

\usepackage{subcaption}

\title{Schur Stability of Matrix Segment via Bialternate Product}

\author{ \href{https://orcid.org/0000-0002-7561-3288}{\includegraphics[scale=0.06]{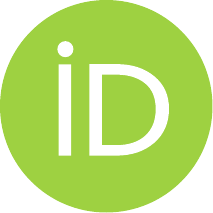}\hspace{1mm}\c{S}erife~Y{\i}lmaz}\\
	Department of Mathematics Education\\
	Mehmet Akif Ersoy University\\
	Burdur, 15030 \\
	\texttt{serifeyilmaz@mehmetakif.edu.tr} \\
}

\renewcommand{\shorttitle}{Schur Stability of Matrix Segment via Bialternate Product}

\hypersetup{
pdftitle={A template for the arxiv style},
pdfsubject={q-bio.NC, q-bio.QM},
pdfauthor={David S.~Hippocampus, Elias D.~Striatum},
pdfkeywords={First keyword, Second keyword, More},
}

\begin{document}
\maketitle

\begin{abstract}
	In this study, the problem of robust Schur stability of $n\times n$ dimensional matrix segments by using the bialternate product of matrices is considered. It is shown that the problem can be reduced to the existence of negative eigenvalues of two of three specially constructed matrices and the existence of eigenvalues belonging to the interval $[1,\infty)$ of the third matrix. A necessary and sufficient condition is given for the convex combinations of two stable matrices with rank one difference to be robust Schur stable. It is shown that the robust stability of the convex hull of a finite number of matrices whose two-by-two differences are of rank $1$ is equivalent to the robust stability of the segments formed by these matrices. Examples of applying the obtained results are given.
\end{abstract}

\keywords{Matrix segment \and Schur stability \and Bialternate product \and  Matrix polytope}

\section{Introduction}
\label{sec1}
An $n\times n$ real matrix $A$ is said to be Schur stable if all eigenvalues of $A$ is contained in the open unit disk in the complex plane, i.e. the spectral radius $\rho(A)<1$ (see \cite{Bhattacharya1995}). This property is essential in the stability theory for discrete-time dynamical systems (see \cite[p. 137]{Horn1991}).

The robust stability analysis of uncertain systems, which can be modeled using matrix segments or matrix polytopes, has attracted significant attention in recent years due to its wide range of applications in control theory \cite{Oliveira2007,Chesi2010}. A polytope of matrices, which is the convex hulls of a finite number of matrices, are established as one of the standard representations of uncertainties involved in state-space models of control systems \cite{Barmish1994,Bhattacharya1995}. When a polytope of matrices formulates the system matrices of uncertain systems, a stability problem of the polytope naturally arises. The problem checking whether all convex combinations of k matrices are stable is NP-hard (see \cite{Gurvits2009}). One generally can not expect extreme point or edge results on the stability of polytope of matrices (see \cite{Mori2000} and references therein). Polytopes of matrices appear in stability problems of linear switched systems as well. For example, the stability of the polytope of matrices is necessary condition for the asymptotical stability of a positive linear switched system defined by a finite number of matrices \cite{Fainshil2009}.

Some more general problems covering this topic concern parameter-dependent matrices. In \cite{Buyukkoroglu2015,Chesi2013}, necessary and sufficient conditions are formulated for the eigenvalues of a parameter-dependent matrix to belong to a certain region $D$ of the complex plane. The stability problem of a matrix polytope is related to the problem of the existence of a common quadratic Lyapunov function (CQLF) for the extreme matrices of this polytope. In \cite{Mason2004}, it is concerned with the CQLF existence problem for a family of two discrete-time LTI systems and a CQLF existence-nonexistence theorems are given for a pair of stable matrices in there.

In \cite{Elsner1998}, using a $(k-1)n\times (k-1)n$ dimensional block matrix, a characterization of the Schur stability of all convex combinations of $n\times n$ dimensional Schur stable matrices $A_1, A_2, \dots, A_k$ is derived. Schur stability of all convex combinations of two complex matrices $A_1$, $A_2$ has been characterized in \cite{Soh1990}. Since the Kronecker product is used, $n^2\times n^2$ dimensional matrices arise in this characterization.

This paper aims to solve this problem with lower dimensional matrices. In addition, we show that obtained results  are valid for the stability domain
\begin{equation} \label{dkumesi}
D=\{(x,y)\in \mathbb{R}^2: \ (x-\delta)^2+y^2<r\},
\end{equation}
where $\delta>0$, $r>0$. As an application of these results, we give a necessary and sufficient condition for the stability of the polytope of matrices whose differences are of rank one.

Let $A_1, A_2 \in \mathbb{R}^{n \times n}$ be Schur stable matrices. If each matrix from the segment
\[
  [A_1,A_2]=\{C(\alpha)=\alpha A_1+(1-\alpha) A_2: \ \alpha\in[0,1]\}.
\]
is Schur stable, the segment $[A_1,A_2]$ is called (robust) stable segment.

\section{Schur stability of all convex combinations of two matrices}\label{sec2}

Let $A$ be an $n\times n$ Schur matrix and $I$ be the identity matrix. If $\lambda$ an eigenvalue of $A$ then $\lambda+\alpha$ is an eigenvalue of $A+\alpha I$, where $\alpha$ is scalar. If $A$ is Schur matrix, $(I-A)$ and $(I+A)$ are nonsingular. 

The boundary of the Schur stability region is the unit circle $\partial D=\{z \in \mathbb{C}: \ |z|=1\}$. If there is an unstable matrix on a matrix segment with stable endpoints, then according to continuous root dependence \cite[p. 52]{Barmish1994} there exists a matrix in this segment with an eigenvalue on the unit circle. For this reason, we will express the following three lemmas regarding the absence of an eigenvalue on the unit circle. The first two of these are related to whether there are matrices from the segment with an eigenvalue of $\pm1$. The third lemma is about the existence of matrices from the segment with complex eigenvalues on the unit circle.
\begin{lemma} \label{lemma1}
	Let $A_1$ and $A_2$ be Schur matrices. For all $\alpha\in [0,1]$, \\
	$(-1)^n \det[C(\alpha)-I]>0$ if and only if $(I-A_1)(I-A_2 )^{-1}$ has no negative real eigenvalue.
\end{lemma}
\begin{proof}
$(\Leftarrow):$ Assume that $(I-A_1)(I-A_2 )^{-1}$ has no negative real eigenvalue and $\alpha \in (0,1]$, then 
\[
  \begin{array}{lcl}
    (-1)^n \det[C(\alpha)-I]&=& \det[I-C(\alpha)]\\
    &=&\det[I-(\alpha A_1+(1-\alpha) A_2)]\\
     &=&\det[\alpha I+(1-\alpha)I-\alpha A_1-(1-\alpha) A_2] \\
     &=&\det[\alpha(I-A_1)+(1-\alpha)(I-A_2)] \\
     &=&\det[\alpha\left((I-A_1)(I-A_2)^{-1}+\frac{1-\alpha}{\alpha} I\right)(I-A_2)] \\
     &=&\alpha^n.\det[I-A_2].\det[(I-A_1)(I-A_2)^{-1}+\frac{1-\alpha}{\alpha} I].
  \end{array}
\]
Here $\det[I-A_2]\not=0$ because the matrix $A_2$ is Schur stable. Since $(I-A_1) (I-A_2)^{-1}$ has no negative eigenvalue, $\det[(I-A_1)(I-A_2)^{-1}+\frac{1-\alpha}{\alpha} I]\not=0$ for each $\alpha\in (0,1]$ ($(1-\alpha)/\alpha \in(0,\infty)$ $\Leftrightarrow$ $\alpha \in (0,1]$). As a result, 
  \begin{equation} \label{dnksifir}
    (-1)^n\det(C(\alpha)-I)\not=0 \mbox{ for all } \alpha \in (0,1]. 
  \end{equation} 
Since $A_1,A_2$ are Schur stable matrices, the matrices $A_1-I$ and $A_2-I$  are Hurwitz stable (all eigenvalues lie in the open left half of the complex plane). The Routh-Hurwitz stability criteria lead to the positivity of the coefficients of the characteristic polynomials of $A_1-I$ and $A_2-I$ \cite[p. 9]{Barmish1994}. Here $\det[C(0)-I]=\det[A_2-I]$, $\det[C(1)-I]=\det[A_1-I]$. 

The constant terms of the characteristic polynomials of the matrices $A_1-I$ and $A_2-I$ are $(-1)^n \det[A_1-I]>0$, $(-1)^n \det[A_2-I]>0$, respectively. Then
  \begin{equation} \label{dnkpozitif}
    0<(-1)^n\det[C(0)-I], \quad 0<(-1)^n\det[C(1)-I].
  \end{equation}
From the continuity of $\alpha \to (-1)^n\det[C(\alpha)-I]$ and equations \eqref{dnksifir}-\eqref{dnkpozitif}, $(-1)^n\det[C(\alpha)-I]>0$ for each $\alpha\in [0,1]$. 

$(\Rightarrow):$ Assume that $(-1)^n \det[C(\alpha)-I]>0$. Take arbitrary $\alpha \in (0,1]$. Then
\begin{equation} \label{dnk23}
  \begin{array}{lcl}
    0<(-1)^n \det[C(\alpha)-I]&=&\det[I-C(\alpha)]\\
    &=&\alpha^n \det[I-A_2]\det[(I-A_1)(I-A_2)^{-1}-\beta I],
  \end{array}
\end{equation}
where $\beta=-(1-\alpha)/\alpha$. Since the mapping $-(1-\alpha)/\alpha: (0,1] \to (-\infty,0)$ is onto, from \eqref{dnk23} it follows that for all $\beta<0$
\[
  \det[(I-A_1)(I-A_2)^{-1}-\beta I]\not=0.
\]
Therefore $(I-A_1)(I-A_2)^{-1}$ has no negative eigenvalue.
\end{proof}

\begin{corollary} \label{cor21}
  Let $A_1$ and $A_2$ be Schur stable. $C(\alpha)$ has no eigenvalue $\lambda=1$ $\Leftrightarrow$ $(I-A_1)(I-A_2)^{-1}$ has no negative real eigenvalue.
\end{corollary}
\begin{proof}
$(\Leftarrow):$ If $(I-A_1)(I-A_2)^{-1}$ has no negative real eigenvalue then by Lemma \ref{lemma1} $(-1)^n \det[C(\alpha)-I]>0$, consequently $\det[C(\alpha)-I]\not=0$ and $\lambda=1$ is not eigenvalue of $C(\alpha)$ for all $\alpha \in [0,1]$.

$(\Rightarrow):$ If $C(\alpha)$ has no eigenvalue $\lambda=1$ then $\det[C(\alpha)-I]\not=0$ and by the proof of Lemma \ref{lemma1} (see the part of Lemma \ref{lemma1} starting from formula \eqref{dnksifir}) $(-1)^n \det[C(\alpha)-I]>0$ and by Lemma \ref{lemma1} $(I-A_1)(I-A_2)^{-1}$ has no negative real eigenvalue.
\end{proof}

\begin{lemma} \label{lemma2}
	Let $A_1$ and $A_2$ be Schur matrices. For all $\alpha\in [0,1]$, $ \det[C(\alpha)+I]>0$ if and only if $(I+A_1 )(I+A_2)^{-1}$ has no negative real eigenvalue.
\end{lemma}
\begin{proof}
  Define $B_1=-A_1$, $B_2=-A_2$ and in view of $A$ is Schur stable $\Leftrightarrow$ $-A$ is Schur stable by Lemma \ref{lemma1} we have 
\[
  (I-B_1)(I-B_2)^{-1} \mbox{ has no negative real eigenvalue if  and only if}
\]
\[
  (-1)^n \det[\alpha B_1+(1-\alpha)B_2-I]>0
\]
or 
\[
  (I+A_1)(I+A_2)^{-1} \mbox{ has no negative real eigenvalue if and only if}
\]
\[
  (-1)^{2n}\det[\alpha A_1+(1-\alpha)A_2+I]=\det[C(\alpha)+I]>0.
\]

\end{proof}
\begin{corollary} \label{cor22}
  Let $A_1$ and $A_2$ be Schur stable. $C(\alpha)$ has no eigenvalue $\lambda=-1$ $\Leftrightarrow$ $(I+A_1)(I+A_2)^{-1}$ has no negative real eigenvalue.
\end{corollary}

In the following we obtain condition on nonexistence of eigenvalue of $C(\alpha)$ in the set $\Theta=\{z \in \mathbb{C}: \ |z|=1,\ z\not=\pm 1\}$. This condition is given in term of the bialternate product of matrices.

\begin{definition}{\cite{Elsner2011}}
	The bialternate product of matrices $A=[a_{ij}],\, B=[b_{ij}]\in \mathbb{R}^{n \times n}$ is defined to be the matrix $F=A\cdot B$ where the entries of $F$ are given by
	\[
	  f_{ij,kl}=\frac{1}{2}\left(
	  \begin{vmatrix}
        a_{ik} & a_{il}\\
	  	b_{jk} & b_{jl}	  	
	  \end{vmatrix}+
	  \begin{vmatrix}
        b_{ik} & b_{il}\\
	  	a_{jk} & a_{jl}	  	
	  \end{vmatrix}
	  \right)
	\]
	where $(i,j),\, (k,l)\in Q_{2,n}=\{(p,q): \ p<q\}$.
\end{definition}

The dimension of $A\cdot B$ is $d\times d$, where $d=n(n-1)/2$. If the eigenvalues of $A$ are $\lambda_1$, $\lambda_2$, $\dots$, $\lambda_n$ then the eigenvalues of $A\cdot A$ are written $\lambda_i \lambda_j$ where $i=1,2,\dots,n-1$ and $j=i+1,i+2,\dots,n$ (see \cite{Fuller1968,Govarets1999,Elsner2011}). For example, if a matrix $A\in \mathbb{R}^{3\times 3}$ has three eigenvalues $\lambda_1$, $\lambda_2$, $\lambda_3$ then the eigenvalues of $A\cdot A$ are $\lambda_1 \lambda_2$, $\lambda_1 \lambda_3$, $\lambda_2 \lambda_3$.

\begin{lemma} \label{lemma4}
  A matris $A\in \mathbb{R}^{n \times n}$ has an eigenvalue in $\Theta$ if and only if $\nu(A):=\det[I-A\cdot A]=0$.
\end{lemma}
\begin{proof}
  If $\lambda \in \Theta$ is an eigenvalue of $A$ then $\bar{\lambda}\not=\lambda$ and $\bar{\lambda}$ is an eigenvalue of $A$ as well and $\lambda \cdot \bar{\lambda}=|\lambda|^2=1$ is an eigenvalue of $A \cdot A$. Therefore $\nu(A)=0$.
\end{proof}

We have
\[
\begin{array}{lcl}
I-C(\alpha)\cdot C(\alpha)&=& I-(\alpha A_1+(1-\alpha) A_2 )\cdot(\alpha A_1+(1-\alpha) A_2 )\\
&=& I-A_2\cdot A_2-2\alpha(A_1\cdot A_2-A_2\cdot A_2 )\\
&& -\alpha^2 (A_1\cdot A_1+A_2\cdot A_2-2A_2\cdot A_1).
\end{array}
\]
(We have used $(A+B)\cdot(C+D)=A\cdot C+A\cdot D+B\cdot C+B\cdot D$ and $A\cdot B=B\cdot A$).
Denote
\begin{equation} \label{denklem1}
\begin{array}{lcl}
    F_0 &=& I-A_2\cdot A_2,  \\
    F_1 &=&-2(A_1\cdot A_2-A_2\cdot A_2 ),  \\
    F_2 &=&-(A_1\cdot A_1+A_2\cdot A_2-2A_2\cdot A_1). 
\end{array}
\end{equation}
Then
\[
  I-C(\alpha)\cdot C(\alpha)=F_0+\alpha F_1+\alpha^2 F_2.
\]

\begin{lemma}  \label{lemma3}
Let $A_1$ and $A_2$ be Schur matrices. $C(\alpha)$ has no eigenvalue in $\Theta$ for all $\alpha \in [0,1]$ if and only if the $(2d \times 2d)$ dimensional matrix
\[
M=\begin{bmatrix}
0&I_d \\
-F_0^{-1} F_2&-F_0^{-1} F_1
\end{bmatrix}
\]
has no real eigenvalue in $[1,\infty)$.
\end{lemma}
\begin{proof}
Assume that $C(\alpha)$ has no eigenvalue in $\Theta$ for all $\alpha \in [0,1]$. We write 
\[
\begin{array}{lcl}
    \nu (C(\alpha)) &=& \det[I-C(\alpha)\cdot C(\alpha)]  \\
     &=& \det[F_0 +\alpha F_1+ \alpha^2 F_2]  \\
    &=& \det[F_0^{-1}] . \det[I_d+\alpha F_0^{-1} F_1+ \alpha^2 F_0^{-1} F_2].
\end{array}
\]
Therefore, $\nu(C(\alpha))\neq 0$ for all $\alpha \in(0,1]$ if and only if $\det [I_d+\alpha F_0^{-1} F_1+ \alpha^2 F_0^{-1} F_2]\neq 0$ for all $\alpha \in(0,1]$. 

We rewrite the above determinant as follows
\[
\det[I_d+\alpha F_0^{-1} F_1+\alpha^2 F_0^{-1} F_2] =\alpha^{2d}  \det [1/\alpha^2 I_d+1/\alpha F_0^{-1} F_1+F_0^{-1} F_2].
\]
Take $\mu:=1/\alpha$, then
\[
\det[I_d+\alpha F_0^{-1} F_1+\alpha^2 F_0^{-1} F_2]=\mu^{-2d}  \det[\mu^2 I_d+\mu F_0^{-1} F_1+F_0^{-1} F_2].
\]

On the other hand, the characteristic polynomial of the matrix
$Z=\begin{bmatrix}0 & I_d \\-X & -Y\end{bmatrix}$
is
$\det[\lambda I-Z]= \det[\lambda^2 I_r+\lambda Y+X]$ (see \cite[p. 309]{Barmish1994}). Therefore 
\[
\det[\mu I - M]=\det[\mu^2 I_d+\mu F_0^{-1} F_1+F_0^{-1} F_2].
\]
Thus, $\det[I_d+\alpha F_0^{-1} F_1+\alpha^2 F_0^{-1} F_2] \neq0 $ for all $\alpha\in (0,1]$ if and only if $\det[\mu I-M]\neq0$ for all $\mu \in [1,\infty)$ (recall that $\mu=1/\alpha$ and $\alpha \in (0,1]$ $\Leftrightarrow$ $\mu \in [1,\infty)$). This means that the matrix $M$ has no real eigenvalue in $[1,\infty)$.
\end{proof}

Using the above lemmas, we derive the following result.
\begin{theorem} \label{thm1}
Let $A_1$ and $A_2$ be Schur stable matrices. $C(\alpha)$ is Schur stable for all $\alpha \in [0,1]$ if and only if
\begin{itemize}
  \item[\textit{i})] $(I-A_1 ) (I-A_2 )^{-1} $ and $(I+A_1 ) (I+A_2 )^{-1}$ have no negative real eigenvalue,
  \item[\textit{ii})] $M$ has no real eigenvalue in $[1,\infty)$.
\end{itemize}
\end{theorem}
\begin{proof}
$(\Rightarrow):$
Assume that $C(\alpha)$ is Schur stable for all $\alpha\in[0,1]$. Therefore the matrix $C(\alpha)$ has no eigenvalue in $\Theta \cup \{-1,1\}$. Then by Corollary \ref{cor21}, \ref{cor22} and Lemma \ref{lemma3}, $i$) and $ii$) are satisfied.

$(\Leftarrow):$ By the theorem of continuity eigenvalues on a parameter (see \cite[p. 52]{Barmish1994}), there exists continuous functions $\lambda_i:[0,1]\to \mathbb{C}$ ($i=1,2,\dots,n$) such that $\lambda_1 (\alpha)$, $\lambda_2 (\alpha)$, $\dots$, $\lambda_n (\alpha)$ are eigenvalues of $C(\alpha)$. Here $|\lambda_i (0)|<1$ ($i=1,2,\dots,n$), since the matrix $C(0)$ is Schur stable. Suppose that $C(\alpha_*)$ is not Schur stable for some $\alpha_* \in(0,1)$. Therefore, there exists an index $i_0\in\{1,2,\dots,n\}$ such that $|\lambda_{i_0} (\alpha_*)|\geq 1$. In view of the continuity of $\lambda_{i_0} (\alpha)$ with respect to $\alpha$, there must exists $\tilde{\alpha}\in(0,1)$ such that $| \lambda_{i_0} (\tilde{\alpha})|=1$. The matrix $C(\tilde{\alpha})$ has an eigenvalue which lies on the unit circle. If the real eigenvalue is $1$ this contradicts Corollary \ref{cor21} and $i$). If the real eigenvalue is $-1$, this contradicts Corollary \ref{cor22} and $i$). If the eigenvalue is complex, this contradicts $ii$). These contradictions show that $C(\alpha)$ is stable for all $\alpha \in [0,1]$. 
\end{proof}
\begin{example} \label{yenironek}
  Consider the matrix segment $[A_1,A_2]$ with Schur stable matrices
\[
    A_1=\begin{bmatrix}0.1&-0.2&0.4\\-0.2&0.3&0.6\\-0.3&0.2&0.1\end{bmatrix}
    \mbox{ and }
    A_2=\begin{bmatrix}0.3&0.5&0.2\\0.6&0.1&-0.6\\-0.3&-0.2&0.4\end{bmatrix}.
\]
The eigenvalues of the matrices $(I-A_1)(I-A_2)^{-1}$ and $(I+A_1)(I+A_2)^{-1}$ are $13.4$, $0.2$, $1.2$ and $2.5$, $0.8$, $0.4$ respectively. From \eqref{denklem1}, we have
\[
  \begin{array}{c}
    F_0=\begin{bmatrix}1.27&0.3&0.32\\-0.09&0.82&-0.24\\0.09&-0.06&1.08\end{bmatrix}, \ 
    F_1=\begin{bmatrix}-0.86&-0.52&-0.96\\0.05&0.11&0.47\\-0.46&0.14&-0.53\end{bmatrix},\ \\
    F_2=\begin{bmatrix}0.6&0.08&0.88\\0.08&-0.06&-0.13\\0.32&-0.24&0.54\end{bmatrix}.
  \end{array}
\]
The eigenvalues of matrix
\[
M=\begin{bmatrix}
0&I_3 \\
-F_0^{-1} F_2&-F_0^{-1} F_1
\end{bmatrix}=
\begin{bmatrix}
0&0&0&1&0&0\\
0&0&0&0&1&0\\
0&0&0&0&0&1\\
-0.35&-0.15&-0.56&0.55&0.48&0.74\\
-0.21&0.12&-0.03&0.11&-0.13&-0.37\\
-0.27&0.24&-0.45&0.38&-0.17&0.40
\end{bmatrix}
\]
are $-0.81$, $0.33$, $0.54\pm0.75i$, $0.10\pm0.39i$. Hence $M$ has no real eigenvalue in $[1,\infty)$. By Theorem \ref{thm1}, the segment $[A_1,A_2]$ is robust stable.
\end{example}
\begin{example} \label{example38}
For the Schur matrices
\[
    A_1=\begin{bmatrix}-21.456&-28.539&-26.541\\12.582&16.758&15.552\\2.808&3.627&3.663\end{bmatrix}
    \mbox{ and }
    A_2=\begin{bmatrix}-0.2394&-1.1466&-2.9484\\1.89&3.15&3.15\\-2.9106&-3.8934&-1.4616\end{bmatrix},
\]
we consider the matrix segment $[A_1,A_2]$. The eigenvalues of the matrix $(I-A_1)(I-A_2)^{-1}$ are calculated as $-10.3584$, $-0.4883$ and $7.6502$. Since the matrix has at least one negative eigenvalue, $C(\alpha)=\alpha A_1+(1-\alpha) A_2$ has eigenvalue $\lambda=1$ for some $\alpha \in (0,1)$ by Lemma \ref{lemma1}. Similarly, the eigenvalues of matrix $(I+A_1)(I+A_2)^{-1}$ are $-4.5996$, $-0.01153$, $0.4870$, $C(\alpha)=\alpha A_1+(1-\alpha) A_2$ has eigenvalue $\lambda=-1$ for some $\alpha \in (0,1)$ by Lemma \ref{lemma1} (see Fig. \ref{figure1:sub1} ). As a result there is a Schur unstable matrix in the segment $[A_1,A_2]$  by Theorem \ref{thm1} (see Fig. \ref{figure1:sub2}).
\end{example}
\begin{figure}[h]
\centering
\begin{subfigure}{.5\textwidth}
  \centering
  \includegraphics[width=5.3cm]{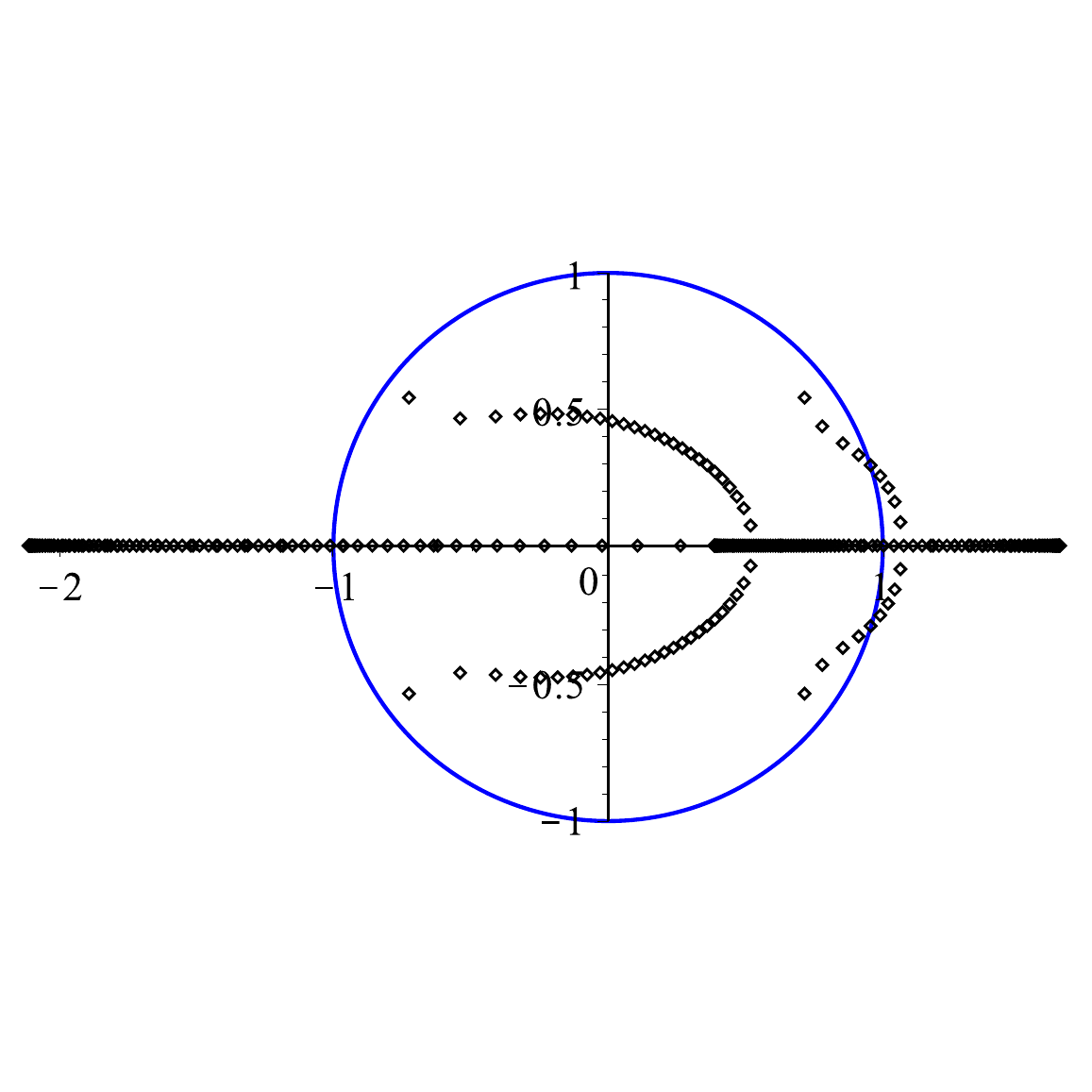}
  \caption{Eigenvalues of the $C(\alpha)$ for $\alpha \in [0,1]$. }
  \label{figure1:sub1}
\end{subfigure}%
\begin{subfigure}{.5\textwidth}
  \centering
  \includegraphics[width=7.4cm]{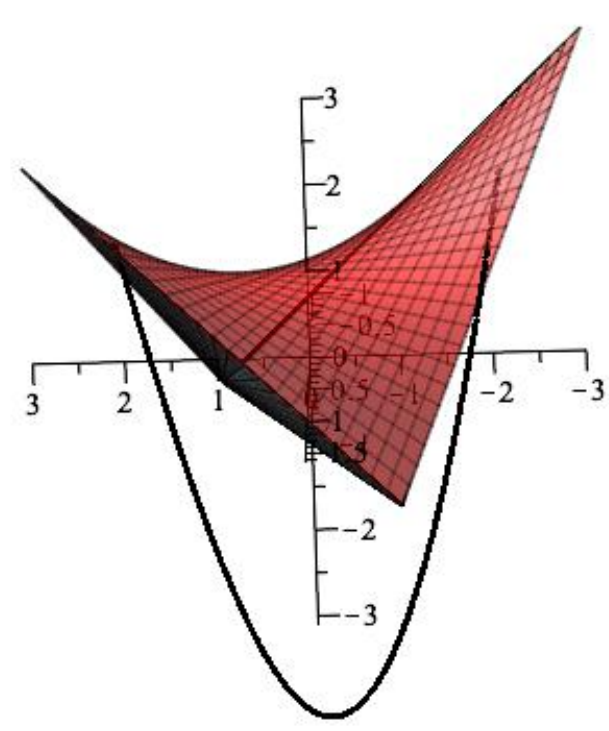}
  \caption{The set of coefficient vectors of third order monic Schur stable polinomials and the curve corresponding to the coefficient vectors of the characteristic polynomials of $C(\alpha)$ for $\alpha \in [0,1]$.}
  \label{figure1:sub2}
\end{subfigure}
\caption{Example \ref{example38}}
\label{figure1}
\end{figure}

Let $A$ be a Metzler matrix that is all off-diagonal elements of $A$ are nonnegative. Then there exists $\alpha \in \mathbb{R}$ and a nonnegative matrix $P$ such that $A=P-\alpha I$. By Perron's theorem the spectral radius $\rho(P)$ is an eigenvalue of $P$, therefore $\rho(P)-\alpha$ is an eigenvalue of $A$. Consequently if $A$ or $-A$ is Metzler matrix then it has a real eigenvalue.

If $\lambda$ is a complex eigenvalue of $A$ then the complex conjugate $\bar{\lambda}$ is also an eigenvalue. Consequently, if $A$ or $-A$ is $2\times 2$ Metzler matrix then it has no complex eigenvalues. Additionally, if $A_1$ and $A_2$ are $2\times 2$ Metzler then $C(\alpha)$ are Metzler for all $\alpha \in [0,1]$.

Summarizing above we have
\begin{theorem} \label{teorem26}
  Let $A_1$ and $A_2$ are $2\times 2$ Metzler Schur stable (or $-A_1$ and $-A_2$ are $2\times 2$ Metzler Schur stable) matrices. The convex combination $C(\alpha)$ is Schur stable for all $\alpha \in [0,1]$ if and only if $(I-A_1 ) (I-A_2 )^{-1} $ and $(I+A_1 ) (I+A_2 )^{-1}$ have no negative real eigenvalue.
\end{theorem}
\begin{proof}
By the aboves, condition $ii$) in Theorem \ref{thm1} is redundant.
\end{proof}

Consider the stability region $D$ \eqref{dkumesi}. Denote
\begin{equation} \label{denklem2}
\begin{array}{lcl}
    \tilde{F}_0 &=& \displaystyle I-A_2\cdot A_2+2\delta A_2\cdot I_n-\delta^2 I_n \cdot I_n,  \\
     \tilde{F}_1 &=&-2(A_1\cdot A_2-A_2\cdot A_2 )+2\delta(A_1\cdot I_n-A_2\cdot I_n),  \\
     \tilde{F}_2 &=&-(A_1\cdot A_1+A_2\cdot A_2-2A_2\cdot A_1). 
\end{array}
\end{equation}
\begin{theorem}
Let $A_1$ and $A_2$ be $D$-stable matrices (all eigenvalues lie in $D$). $C(\alpha)$ is $D$-stable for all $\alpha \in [0,1]$ if and only if
\begin{itemize}
  \item[\textit{i})] $\left[(\delta+r)I-A_1\right] \left[(\delta+r)I-A_2\right]^{-1} $ and $\left[(\delta-r)I+A_1\right]\left[(\delta-r)I+A_2 \right]^{-1}$ have no negative real eigenvalue,
  \item[\textit{ii})] $M=\begin{bmatrix}0 & I\\- \tilde{F}_0^{-1}  \tilde{F}_2&- \tilde{F}_0^{-1}  \tilde{F}_1  \end{bmatrix}$ has no real eigenvalue in $[1,\infty)$.
\end{itemize}
\end{theorem}
\begin{proof}
The proof is immediate.
\end{proof}

\section{Robust stability of a one matrix polytope}\label{sec3}

In this section using results of Section \ref{sec2} we obtain condition for robust stability of a matrix polytope with rank one uncertainty, namely consider a polytope
\begin{equation}\label{dnkl31}
  \mathcal{A}=\mathrm{co}\{A_1,A_2,\dots,A_N\}, \quad A_i=B_0+B_i \quad (i=1,2,\dots,N)
\end{equation}
where $B_0 \in \mathbb{R}^{n\times n}$, $\mathrm{rank}(B_i)=1$. For $A\in \mathbb{R}^{n\times n}$ it is well known that
\[
  \begin{array}{lcl}
    \mathrm{rank}(A)=1 &\Leftrightarrow& A=bc^T \mbox{for some nonzero column vectors }b,c \in \mathbb{R}^n.\\
  \end{array}
\]
Therefore we will assume that
\[
  B_i=bc_i^T \mbox{ or } B_i=b_ic^T \ (i=1,2,\dots,N).
\]
\begin{lemma}\label{lemma31}
  For $A \in \mathbb{R}^{n\times n}$, if $\mathrm{rank}(A)=1$ then the eigenvalues of $A$ are zero with multiplicity $(n-1)$ and $\mathrm{trace}(A)$ with multiplicity $1$.
\end{lemma}
\begin{proof}
  Assume that $A=uv^T$ for nonzero $u,v \in \mathbb{R}^n$. The equation $v^Tx=0$ has $(n-1)$ linear independent solutions $x^1,x^2,\dots,x^{n-1}\in \mathbb{R}^n$. Then $Ax^i =0.x^i$ $(i=1,2,\dots,n-1)$ and $\lambda_1=\lambda_2=\dots=\lambda_{n-1}=0$. On the other hand from the well known equality
\[
  \lambda_1+\lambda_2+\cdots+\lambda_n=\mathrm{trace}(A)
\]
it follows that $\lambda_n=\mathrm{trace}(A)$.
\end{proof}
\begin{lemma}\label{lemma32}
  If $A=uv^T$ ($u,v\in \mathbb{R}^n$) then $A\cdot A=0$.
\end{lemma}
\begin{proof}
  \[
    A=\begin{bmatrix}u_1\\ u_2\\ \vdots \\ u_n \end{bmatrix}
        \begin{bmatrix}v_1 & v_2 & \cdots & v_n \end{bmatrix}=
        \begin{bmatrix}u_1v_1 & u_1v_2 & \cdots & u_1v_n\\u_2v_1 & u_2v_2 & \cdots & u_2v_n\\\vdots & \vdots & \ddots & \vdots\\u_nv_1 & u_nv_2 & \cdots & u_nv_n\end{bmatrix}.
  \]
The entries of the matrix $A\cdot A$ are
\[
  \begin{array}{lcl}
    f_{ij,kl}&=&\displaystyle \frac{1}{2}\left(\begin{vmatrix}
          u_iv_k & u_iv_l\\
	  u_jv_k & u_jv_l	  	
	\end{vmatrix}+
	\begin{vmatrix}
          u_iv_k & u_iv_l\\
	  u_jv_k & u_jv_l	  	
	\end{vmatrix}
	\right)=
	\begin{vmatrix}
          u_iv_k & u_iv_l\\
	  u_jv_k & u_jv_l	  	
	\end{vmatrix}\\
	&=&
	u_iu_j
	\begin{vmatrix}
          v_k & v_l\\
	  v_k & v_l	  	
	\end{vmatrix}=0.
    \end{array}
  \]
  Therefore $A\cdot A=0$.
\end{proof}
\begin{lemma}\label{lemma33}
  Let a segment $[A_1,A_2]$ be given, where $A_1$ and $A_2$ are Schur stable, $\mathrm{rank}(A_1-A_2)=1$ with $A_1=A_0+ba^T$, $A_2=A_0+bc^T$. Then $[A_1,A_2]$ is robustly Schur stable if and only if
\begin{itemize}
  \item[\textit{i})] the matrices $(I-A_1)(I-A_2 )^{-1}$ and $(I+A_1)(I+A_2)^{-1}$ have no real negative eigenvalue,
  \item[\textit{ii})] $-F_0^{-1}F_1$ has no real eigenvalue in $[1,\infty)$, where $F_0$ and $F_1$ are defined in \eqref{denklem1}.
\end{itemize}
\end{lemma}
\begin{proof}
  \[
    \begin{array}{lcl}
      A_1\cdot A_1&=&A_0\cdot A_0+2A_0\cdot (ba^T )+(ba^T )\cdot (ba^T ),\\
      A_2\cdot A_2&=&A_0\cdot A_0+2A_0\cdot (bc^T )+(bc^T )\cdot (bc^T ),\\
      A_1\cdot A_2&=&A_0\cdot A_0+A_0\cdot (bc^T )+A_0\cdot (ba^T )+(ba^T )\cdot (bc^T ).
    \end{array}
  \]
  From \eqref{denklem1},
  \[
    \begin{array}{lcl}
      -F_2&=&A_1\cdot A_1+A_2\cdot A_2-2A_2\cdot A_1\\
             &=&(ba^T )\cdot (ba^T )-2(ba^T )\cdot (bc^T )+(bc^T )\cdot (bc^T )\\
             &=&(ba^T-bc^T )\cdot (ba^T-bc^T )\\
             &=&[b(a^T-c^T )]\cdot [b(a^T-c^T )]=0
    \end{array}
  \]
  by Lemma \ref{lemma32}. Therefore the matrix $M$ in Lemma \ref{lemma3} becomes
  \[
    M=\begin{bmatrix} 0 & I_d\\ 0 & -F_0^{-1}F_1 \end{bmatrix} 
  \]
  The eigenvalues of $M$ consist of $0$ (with multiplicity $d$) and the eigenvalues of the matrix $-F_0^{-1}F_1$. Then the necessary and sufficient condition for Schur stability of $\mathcal{A}$ follows from Theorem \ref{thm1}. 
\end{proof}
\begin{lemma}\label{lemma34}
  If $A=uv^T$ $(u,v\in\mathbb{R}^n)$ then
  \[
    \mathrm{trace}(A)=1-\det(I-A).
  \]
\end{lemma}
\begin{proof}
  The eigenvalues of $A$ are $\lambda_1=\cdots=\lambda_{n-1}=0$, $\lambda_n=\mathrm{trace}(A)$ (Lemma \ref{lemma31}). The eigenvalues of $I-A$ are $\mu_1=\cdots=\mu_{n-1}=1$, $\mu_n=1-\mathrm{trace}(A)$. The determinant of any matrix equals to the product of the eigenvalues, therefore $\det(I-A)=1-\mathrm{trace}(A)$.
\end{proof}
\begin{lemma}\label{lemma35}
  Consider the polytope $\mathcal{A}$ defined by \eqref{dnkl31} and let $A$ be any matrix from $\mathcal{A}$:
\[
  A=\alpha_1A_1+\cdots+\alpha_N A_N, \ \sum_{i=1}^N \alpha_i=1, \ \alpha_i \geq 0.
\]
Then
\[
  p_A(s)=\alpha_1p_{A_1}(s)+\alpha_2p_{A_2}(s)+\cdots+\alpha_Np_{A_N}(s)
\] 
where $p_A(s)$ is the characteristic polynomial of $A$.
\end{lemma}
\begin{proof}
  $A=B_0+(\alpha_1B_1+\alpha_2B_2+\cdots+\alpha_NB_N)$,
  \[
    \begin{array}{lcl}
      p_A(s)&=&\det[sI-B_0-(\alpha_1B_1+\alpha_2B_2+\cdots+\alpha_NB_N)]\\
            &=&\det\left[(sI-B_0)(I-(sI-B_0)^{-1})(\sum_{i=1}^N \alpha_ibc_i^T)\right]\\
            &=&\det[sI-B_0]\det[I-\tilde{b}(\tilde{c})^T], \ \left(\tilde{b}=(sI-B_0)^{-1}b,\ \tilde{c}=\left(\sum_{i=1}^N \alpha_ic_i\right)^T\right).
    \end{array}
  \]
  Using Lemma \ref{lemma34}
  \[
    \begin{array}{lcl}
      p_A(s)&=&\displaystyle \det [(sI-B_0) (1-\mathrm{trace}\,\tilde{b}(\tilde{c})^T) ]\\
      &=&\displaystyle \det[(sI-B_0)(1-\sum_{i=1}^N\alpha_i\mathrm{trace}\, \tilde{b}\tilde{c_i}^T)]\\
      &=&\displaystyle \det[(sI-B_0)(1-\sum_{i=1}^N \alpha_i(1-\det[I-\tilde{b}c_i^T]))]\\
      &=&\displaystyle \det[(sI-B_0)(1-1+\sum_{i=1}^N \alpha_i \det[I-\tilde{b}c_i^T])]\\
      &=&\displaystyle \sum_{i=1}^N \alpha_i \det[sI-B_0] \det[I-(sI-B_0)^{-1}bc_i^T]\\
      &=&\displaystyle \sum_{i=1}^N \alpha_i \det\left[(sI-B_0)(I-(sI-B_0)^{-1}bc_i^T)\right]\\
      &=&\displaystyle \sum_{i=1}^N \alpha_i \det[sI-B_0-bc_i^T]\\
      &=&\displaystyle \sum_{i=1}^N \alpha_i p_{A_i}(s)
    \end{array}
  \]
\end{proof}
Now we arrive at the main result of this section.
\begin{theorem} \label{teorem36}
  Let the family \eqref{dnkl31} be given with Schur stable generators $A_i$. This family is robustly Schur stable if and only if
  \begin{itemize}
    \item[\textit{i})] $(I-A_i)(I-A_j )^{-1}$ and $(I+A_i)(I+A_j)^{-1}$ have no negative real eigenvalues ($i,j=1,2,\dots,N$; $i<j$),
    \item[\textit{ii})] $-{\left(F^{ij}_0\right)}^{-1}F^{ij}_1$ have no real eigenvalue in $[1,\infty)$ ($i,j=1,2,\dots,N$; $i<j$), where $F_0^{ij}$ and $F_1^{ij}$ are defined by \eqref{denklem1} with replacing $A_1$ by $A_i$ and $A_2$ by $A_j$ respectively.
  \end{itemize}
\end{theorem}
\begin{proof}
  By Lemma \ref{lemma35} the family of the characteristic polynomials of the matrix family $\mathcal{A}$ \eqref{dnkl31} is the polynomial polytope
  \[
    \mathcal{P}=\mathrm{co}\{p_{A_1}(s),\dots,p_{A_N}(s)\},
  \]
  and $\mathcal{A}$ is robustly stable if and only if $\mathcal{P}$ is robustly stable. By the Edge Theorem (\cite[p. 153]{Barmish1994}) the polytope $\mathcal{P}$ is robustly stable if and only if all the edges
  \[
    \begin{array}{lcl}
      [p_{A_i}(s),p_{A_j}(s)]&=&\{\alpha p_{A_i}(s)+(1-\alpha) p_{A_j}(s): \ \alpha \in [0,1]\}\\
      &=&\mathrm{co}\{p_{A_i}(s),p_{A_j}(s)\} \quad (i,j=1,2,\dots,N; \ i<j)
    \end{array}
  \]
  are stable. The polynomial segment $[p_{A_i}(s),p_{A_j}(s)]$ is the set of the characteristic polynomials of the matrix segment $[A_i,A_j]$. By Lemma \ref{lemma33} the segment $[A_i,A_j]$ is stable if and only if $i$) and $ii$) are satisfied.
  
  Summarizing; $\mathcal{A}$ is stable $\Leftrightarrow \ \mathcal{P}$ is stable $\Leftrightarrow$ All edges $[p_{A_i}(s),p_{A_j}(s)]$ are stable $\Leftrightarrow$ All matrix segment $[A_i,A_j]$ are stable $\Leftrightarrow$ $i$) and $ii$) are satisfied.
 \end{proof}

\begin{example}
For the vectors $b=(1,-1,1)^T$, $c_1=(0.5,0.5,-0.5)^T$, $c_2=(-0.25,0.5,0.5)^T$, $c_3=(0.1,-0.1,-0.1)^T$ and matrix
\[
  B_0=\begin{bmatrix}-0.3&0.3&-0.3\\ -0.1&0.1&-0.1\\ 0.2&-0.2&0.2 \end{bmatrix}
\]
  consider the polytope $\mathcal{A}=\mathrm{co}\{A_1,A_2,A_3\}$ where
  \[
    \begin{array}{l}
      	A_1=B_0+b\, c_1^T=\begin{bmatrix} 0.2&0.8&-0.8\\-0.6&-0.4&0.4\\0.7&0.3&-0.3\end{bmatrix},\\ \\
      	A_2=B_0+b\, c_2^T=\begin{bmatrix} -0.55&0.8&0.2\\0.15&-0.4&-0.6\\-0.05&0.3&0.7\end{bmatrix},\\ \\
      	A_3=B_0+b\, c_3^T=\begin{bmatrix} -0.2&0.2&-0.4\\-0.2&0.2&0\\0.3&-0.3&0.1\end{bmatrix}
    \end{array}
  \]
are Schur stable matrices. The matrices $(I-A_i)(I-A_j)^{-1}$ and $(I+A_i)(I+A_j)^{-1}$ have no negative real eigenvalue ($i,j=1,2,3$, $i<j$). We calculate
  \[
    \begin{array}{ll}
    	F_0^{12}=\begin{bmatrix}0.9&-0.3&0.4\\0.125&1.375&-0.5\\-0.025&-0.075&1.1\end{bmatrix},&
    	F_1^{12}=\begin{bmatrix}-0.3&0.7&-0.4\\0.375&-0.875&0.5\\-0.075&0.175&-0.1\end{bmatrix},\\ \\
	F_0^{13}=\begin{bmatrix}1&0.08&-0.08\\0&0.9&0.1\\0&0.02&0.98\end{bmatrix},&
    	F_1^{13}=\begin{bmatrix}-0.4&0.32&0.08\\0.5&-0.4&-0.1\\-0.1&0.08&0.02\end{bmatrix},\\ \\
	F_0^{23}=\begin{bmatrix}1&0.08&-0.08\\0&0.9&0.1\\0&0.02&0.98\end{bmatrix},&
    	F_1^{23}=\begin{bmatrix}-0.1&-0.38&0.48\\0.125&0.475&-0.6\\-0.025&-0.095&0.12\end{bmatrix}.
    \end{array}
  \]
  The eigenvalues of $-(F_{0}^{ij})^{-1}F_1^{ij}$ are not in $[1,\infty)$ ($i,j=1,2,3$, $i<j$). By Theorem \ref{teorem36}, all matrices in the family $\mathcal{A}$ are Schur stable.
\end{example}
\section{Conclusion}
In this study, we have investigated the robust Schur stability of $n\times n$ dimensional matrix segments using the bialternate product of matrices. We have shown that the problem can be reduced to checking the existence of negative eigenvalues in two out of three specially constructed matrices and the presence of eigenvalues in the interval $[1,\infty)$ for the third matrix.

We have provided a necessary and sufficient condition for the convex combinations of two stable matrices with rank one difference to be robustly Schur stable. Furthermore, we have demonstrated that the robust stability of the convex hull of a finite number of matrices, whose pairwise differences are of rank $1$, is equivalent to the robust stability of the segments formed by these matrices.

The obtained results have been illustrated through examples, showcasing their applicability in analyzing the stability of matrix polytopes with rank one uncertainty. These findings can be particularly useful in the stability analysis of linear systems subject to polytopic uncertainties, which are commonly encountered in control systems.

\bibliographystyle{unsrtnat}
\bibliography{references}

\begin{thebibliography}{16}
\providecommand{\natexlab}[1]{#1}
\providecommand{\url}[1]{\texttt{#1}}
\expandafter\ifx\csname urlstyle\endcsname\relax
  \providecommand{\doi}[1]{doi: #1}\else
  \providecommand{\doi}{doi: \begingroup \urlstyle{rm}\Url}\fi

\bibitem[Bhattacharya et~al.(1995)Bhattacharya, Chapellat, and
  Keel]{Bhattacharya1995}
Shankar~P. Bhattacharya, Herve Chapellat, and Lee~H. Keel.
\newblock \emph{Robust Control: the parametric approach}.
\newblock Prentice-Hall, New Jersey, 1995.

\bibitem[Horn and Johnson(1991)]{Horn1991}
Roger~A. Horn and Charles~R. Johnson.
\newblock \emph{Topics in matrix analysis}.
\newblock Cambridge University Press, New York, 1991.

\bibitem[Oliveira and Peres(2007)]{Oliveira2007}
Ricardo C. L.~F. Oliveira and Pedro L.~D. Peres.
\newblock Parameter-dependent lmis in robust analysis: characterization of
  homogeneous polynomially parameter-dependent solutions via lmi relaxations.
\newblock \emph{IEEE Transactions on Automatic Control}, 7\penalty0
  (52):\penalty0 1334--1340, July 2007.
\newblock \doi{10.1109/TAC.2007.900848}.

\bibitem[Chesi(2010)]{Chesi2010}
Graziano Chesi.
\newblock Lmi techniques for optimization over polynomials in control: a
  survey.
\newblock \emph{IEEE Transactions on Automatic Control}, 11\penalty0
  (55):\penalty0 2500--2510, November 2010.
\newblock \doi{10.1109/TAC.2010.2046926}.

\bibitem[Barmish(1994)]{Barmish1994}
B.~Ross Barmish.
\newblock \emph{New Tools for Robustness of Linear Systems}.
\newblock Macmillan, New York, 1994.

\bibitem[Gurvits and Olshevsky(2009)]{Gurvits2009}
Leonid Gurvits and Alexander Olshevsky.
\newblock On the np-hardness of checking matrix polytope stability and
  continuous-time switching stability.
\newblock \emph{IEEE Transactions on Automatic Control}, 54\penalty0
  (2):\penalty0 337--341, February 2009.
\newblock \doi{10.1109/TAC.2008.2007177}.

\bibitem[Mori and Kokame(2000)]{Mori2000}
Takehiro Mori and Hideki Kokame.
\newblock A parameter-dependent lyapunov function for a polytope of matrices.
\newblock \emph{IEEE Transactions on Automatic Control}, 45\penalty0
  (8):\penalty0 1516--1519, August 2000.
\newblock \doi{10.1109/9.871762}.

\bibitem[Fainshil et~al.(2009)Fainshil, Margaliot, and Chigansky]{Fainshil2009}
Lior Fainshil, Michael Margaliot, and Pavel Chigansky.
\newblock On the stability of positive linear switched systems under arbitrary
  switching laws.
\newblock \emph{IEEE Transactions on Automatic Control}, 54\penalty0
  (4):\penalty0 897--899, March 2009.
\newblock \doi{10.1109/TAC.2008.2010974}.

\bibitem[B\"{u}y\"{u}kk\"{o}ro\u{g}lu et~al.(2015)B\"{u}y\"{u}kk\"{o}ro\u{g}lu,
  \c{C}elebi, and Dzhafarov]{Buyukkoroglu2015}
Taner B\"{u}y\"{u}kk\"{o}ro\u{g}lu, G\"{o}khan \c{C}elebi, and Vakif Dzhafarov.
\newblock On the robust stability of polynomial matrix families.
\newblock \emph{Electronic Journal of Linear Algebra}, 30:\penalty0 905--915,
  February 2015.
\newblock \doi{10.13001/1081-3810.3093}.

\bibitem[Chesi(2013)]{Chesi2013}
Graziano Chesi.
\newblock Exact robust stability analysis of uncertain systems with a scalar
  parameter via lmis.
\newblock \emph{Automatica}, 49\penalty0 (4):\penalty0 1083--1086, April 2013.
\newblock \doi{10.1016/j.automatica.2013.01.033}.

\bibitem[Mason and Shorten(2004)]{Mason2004}
Oliver Mason and Robert Shorten.
\newblock On common quadratic lyapunov functions for stable discrete-time lti
  systems.
\newblock \emph{IMA Journal of Applied Mathematics}, 69\penalty0 (3):\penalty0
  271--283, June 2004.
\newblock \doi{10.1093/imamat/69.3.271}.

\bibitem[Elsner and Szulc(1998)]{Elsner1998}
Ludwig Elsner and Tomasz Szulc.
\newblock Convex combinations of matrices-nonsingularity and schur stability
  characterizations.
\newblock \emph{Linear and Multilinear Algebra}, 44\penalty0 (4):\penalty0
  301--312, 1998.
\newblock \doi{10.1080/03081089808818566}.

\bibitem[Soh(1990)]{Soh1990}
Cheong~B. Soh.
\newblock Schur stability of convex combinations of matrices.
\newblock \emph{Linear Algebra and its Applications}, 128:\penalty0 159--168,
  January 1990.
\newblock \doi{10.1016/0024-3795(90)90290-S}.

\bibitem[Elsner and Monov(2011)]{Elsner2011}
Ludwig Elsner and Viladimir Monov.
\newblock The bialternate matrix product revisited.
\newblock \emph{Linear Algebra and its Applications}, 434\penalty0
  (4):\penalty0 1058--1066, February 2011.
\newblock \doi{10.1016/j.laa.2010.10.016}.

\bibitem[Fuller(1968)]{Fuller1968}
A.~Thomas Fuller.
\newblock Conditions for a matrix to have only characteristic roots with
  negative real parts.
\newblock \emph{Journal of Mathematical Analysis and Applications}, 23\penalty0
  (1):\penalty0 71--98, July 1968.
\newblock \doi{10.1016/0022-247X(68)90116-9}.

\bibitem[Govaerts and Sijnave(1999)]{Govarets1999}
Willy J.~F. Govaerts and B.~Sijnave.
\newblock Matrix manifolds and jordan structure of the bialternate matrix
  product.
\newblock \emph{Linear Algebra and its Applications}, 292\penalty0
  (1):\penalty0 245--266, May 1999.
\newblock \doi{10.1016/S0024-3795(99)00039-7}.

\end{thebibliography}

\end{document}